\documentclass[12pt,oneside,reqno]{amsart}

\usepackage{amsmath,amscd,amsthm,amsfonts,amssymb}
\usepackage{mathtools,mathrsfs}
\mathtoolsset{showonlyrefs}
\usepackage{comment}
\usepackage{todonotes}
\usepackage{bm}
\usepackage{bbm}
\usepackage{url}
\usepackage{esint}
\usepackage{dsfont}
\usepackage{relsize}
\usepackage{enumitem}
\usepackage{stmaryrd}

\usepackage{color,xcolor}

\usepackage{geometry}
\usepackage{hyperref}

\numberwithin{equation}{section}
\newtheorem{theorem}{Theorem}[section]

\newtheorem{lemma}[theorem]{Lemma}

\newtheorem{remark}[theorem]{Remark}

\newcommand{\R}{\mathbb{R}}

\newcommand{\p}{\partial}

\newcommand{\DN}{\Lambda^T}

\DeclareMathOperator{\tr}{tr}
\DeclareMathOperator{\diam}{diam}
\newcommand{\abs}[1]{\left\lvert#1\right\rvert}

\newcommand{\der}{\mathrm{d}}

\usepackage[framemethod=tikz]{mdframed}
\usepackage{url}

\newmdenv[backgroundcolor=blue!10,innerlinewidth=0.5pt,roundcorner=4pt,linecolor=aliceblue,innerleftmargin=6pt,innerrightmargin=6pt,innertopmargin=6pt,innerbottommargin=6pt]{mybox}

\title{Rigidity of homogeneous Lamé systems}

\author[Ilmavirta]{Joonas Ilmavirta}
\address{J. Ilmavirta, Department of Mathematics and Statistics, University of Jyv\"askyl\"a, Jyv\"askyl\"a, 40014,  Finland} \email{joonas.ilmavirta@jyu.fi}

\author[Saksala]{Teemu Saksala}
\address{T. Saksala, Department of Mathematics\\
North Carolina State University, Raleigh\\ 
NC 27695, USA}
\email{tssaksal@ncsu.edu}

\author[Yan]{Lili Yan}
\address{L. Yan, Department of Mathematics\\
North Carolina State University, Raleigh\\ 
NC 27695, USA}
\email{lyan6@ncsu.edu}
\date{}

\begin{document}

\begin{abstract}
In this short paper, we show that any Lam\'e system whose Dirichlet-to-Neumann map for the elastic wave equation agrees with the one arising from the homogeneous Lam\'e system must actually be homogeneous. We do not need to impose any assumptions for the Lam\'e coefficients that we aim to recover. 
We use the fact that the homogeneous system gives rise to a geometry that is both simple and admits a strictly convex foliation.

\end{abstract}

\maketitle

\section{Uniqueness of Lamé triplets}

We prove that ``if a Lamé system looks homogeneous, then it is homogeneous'':

\begin{theorem}
\label{thm:special}
Let $\Omega \subset \R^3$ be an open, bounded, and strictly convex domain with smooth boundary.
Let $(\lambda_1,\mu_1,\rho_1)$ be constants and let $(\lambda_2,\mu_2,\rho_2)$ smooth functions, both satisfying positivity conditions~\eqref{eq:strong-convexity}.
If the hyperbolic Dirichlet-to-Neumann maps $\DN_1$ and $\DN_2$ associated with these coefficients agree up to time $T>
\sqrt{\frac{\rho_1}{\mu_1}}\diam(\Omega)$, then
\begin{equation}
\lambda_1 = \lambda_2,
\quad 
\mu_1 = \mu_2,
\quad
\text{and}
\quad
\rho_1 = \rho_2.
\end{equation}
\end{theorem}

Theorem~\ref{thm:special} is a corollary of Theorem~\ref {thm:general} stated below because the homogeneous model satisfies the geometric assumptions.
We next define more carefully the concepts appearing in these theorems.
Let $\Omega \subset \R^3$ be a smooth, bounded, and strictly convex domain with smooth boundary $\p\Omega$.
We refer to the \textit{Lamé parameters} $(\lambda,\mu)$ and the density $\rho$ together as a \textit{Lamé triplet}.
We assume the functions $\lambda,\mu,\rho\colon\overline{\Omega}\to\R$ to be smooth and satisfy positivity conditions we will lay out in a moment.

A Lamé triplet $(\lambda,\mu,\rho)$ gives rise to an \textit{isotropic stiffness tensor} field defined pointwise by
\begin{equation}
c_{ijk\ell}
=
\lambda\delta_{ij}\delta_{k\ell}
+
\mu\left(\delta_{ik}\delta_{j\ell}
+
\delta_{i\ell}\delta_{jk}\right).
\end{equation}
Stiffness tensors of this special type depend only on the two Lam\'e parameters $(\lambda,\mu)$ instead of 21 independent parameters of an anisotropic stiffness tensor field, which is required to only satisfy the minimal symmetry conditions
\begin{equation}
c_{ijk\ell}=c_{jik\ell}=c_{ij\ell k}=c_{k\ell ij}.
\end{equation}

The elastic wave operator $P$ is defined by
\begin{equation}
\label{eq:EWO}
(Pu(t,x))_i
=
\rho(x) \partial_t^2 u_i(t,x)
-
\sum_{j,k,\ell}
\partial_{x_j}(c_{ijk\ell}(x) \partial_{x_\ell}u_{k}(t,x)),
\quad \text{ for } i\in \{1,2,3\},
\end{equation}
while acting on sufficiently smooth vector fields $u\colon\R\times\Omega\to\R^3$.

For the elastic wave operator to be hyperbolic, we require that the density $\rho$ is positive and that the spatial operator in equation~\eqref{eq:EWO} is elliptic.
This, as it turns out\footnote{
This ellipticity amounts to requiring that the $n$-dimensional stiffness tensor $c$ is positive definite as a quadratic form on the space of symmetric real $n\times n$ matrices.
That is, $c(A,A)\coloneqq A_{ij}c_{ijk\ell}A_{k\ell}>0$ for all symmetric non-zero $A\in\R^{n\times n}$.
First, take a matrix for which $A_{12}=A_{21}=1$ and all other entries vanish.
The requirement in this case is equivalent to $\mu>0$.
Let us then consider a general $A$.
Using the Cauchy--Schwarz inequality on the vector of eigenvalues of $A$ shows that $\tr(A)^2\leq n\tr(A^2)$.
Using the Hilbert–Schmidt norm, we thus get $c(A,A)=\lambda\tr(A)^2+2\mu\tr(A^2)\geq(n\lambda+2\mu)\abs{A}^2$, with equality when $A$ is a multiple of the identity.
This gives the other condition $n\lambda+2\mu>0$.
In the case $n=1$, the condition $\mu>0$ becomes unnecessary, but the condition $\lambda+2\mu>0$ remains correct.
}, is equivalent to
\begin{equation}
\label{eq:strong-convexity}
\begin{cases}
\rho>0
,\\
\mu>0
,\\
3 \lambda+2\mu>0
\end{cases}
\end{equation}
at all points in $\overline{\Omega}$. We also note that the elastic wave operator $P$, arising from isotropic elasticity, is of real principal type as defined for instance in \cite[Definition 3.1]{dencker1982propagation}. We refer to \cite[Section 4.1]{rachele2000inverse} for the proof of this claim.

Given a sufficiently regular function $f\colon(0,T)\times\partial\Omega\to\R^3$, we consider the hyperbolic boundary value problem
\begin{equation}
\label{eq:bvp}
\begin{cases}
Pu = 0 & \text{in } Q^T:=(0,T)\times\Omega,\\
u = f & \text{on } \Sigma^T:=(0,T)\times \partial \Omega, \\
u(0,x)=\partial_t u(0,x) = 0 & \text{in }\Omega.
\end{cases}
\end{equation}
We denote the Neumann boundary value of the solution $u$ by $\DN f\colon \Sigma^T \to\R^3$, defined by
\begin{equation}
(\DN f)_i
=
\nu_jc_{ijk\ell}\p_{x_\ell} u_k
\quad \text{ for } i\in \{1,2,3\}.
\end{equation}
Above and in what follows, we adopt the convention of summing over the repeated indices. 
Physically, we may think of $f$ as the surface displacement and $\DN f$ as the resulting traction (force) on the surface. 

Well-posedness of the initial boundary value problem~\eqref{eq:bvp} and thus well-definedness of the Dirichlet-to-Neumann map (DN-map) $\DN$ for $f\in L^2(\Sigma^T;\R^3)$ is given in~\cite[Theorem 1]{belishev2002dynamical}. In particular, for each $f \in L^2(\Sigma^T;\mathbb{R}^3)$, there exists a
unique solution
$u \in C([0,T]; L^2(\Omega;\mathbb{R}^3))$.

The hyperbolic system $Pu=0$ has two wave speeds.
Pressure waves (or $p$-waves or primary waves or longitudinal waves) travel at the speed $c_p=\sqrt{\frac{\lambda+2\mu}{\rho}}$ 
and shear waves (or $s$-waves or secondary waves or transversal waves) travel at the speed
$c_s=\sqrt{\frac{\mu}{\rho}}$.
We always have $c_p>c_s>0$ due to~\eqref{eq:strong-convexity}.
The singularities of elastic waves follow the geodesics of the conformally Euclidean metrics
$g_p=c_p^{-2}\delta$ and $g_s=c_s^{-2}\delta$,
where $\delta$ is the Euclidean metric tensor.
See~\cite[Theorem~4.2]{dencker1982propagation} for general propagation of singularities for a hyperbolic system of real principal type and \cite[Proposition 4.1]{rachele2000inverse} for the specific isotropic elastic case.

In our main theorem below, we consider $i=1$ to be a well-behaved and known model and $i=2$ to be an \textit{unknown model with no geometric assumptions whatsoever}.
We show that if $\DN_1=\DN_2$, then in fact the two models agree.
Our proof of the homogeneous case of Theorem~\ref {thm:special} is based on the fact that the Euclidean metric satisfies the geometric assumptions of Theorem~\ref {thm:general}. These two assumptions are the \textit{simplicity} of the metric, the domain is strictly convex and between any two points there is a smoothly varying unique distance minimizing geodesic (see for instance, \cite[Section 3.8.]{paternain2023book} for more details on simple Riemannian metrics), and the \textit{strictly convex foliation condition} i.e. the existence of a smooth strictly convex function in a sense that the restriction of this function on each geodesic is strictly convex (see for instance, \cite[Section 2]{paternain2019geodesic} for more details on the foliation condition for Riemannian metrics). Both of these conditions are stable over small $C^2$-variations of the metric.

The existence of a strictly convex function does not imply simplicity, as the former condition allows the existence of conjugate points while the latter forbids it. It is an open problem whether simple manifolds admit a strictly convex function in dimension $n\geq3$. 

\begin{theorem}
\label{thm:general}
Let $\Omega \subset \R^3$ be an open and bounded domain with smooth boundary.
Suppose that the Dirichlet-to-Neumann maps $\DN_1$ and $\DN_2$ for the boundary value  problem~\eqref{eq:bvp} corresponding to two smooth Lam\'e triplets $(\lambda_i,\mu_i,\rho_i)$ with $i=1,2$ satisfy the following:
\begin{itemize}
\item[(a)] 
$\DN_1=\DN_2$.
\item[(b)]
$T>\diam_{g_{1,s}}(\Omega)$ in the sense of the distance induced by $g_{1,s}$ on $\overline{\Omega}$.
\item[(c)]
Both Lamé triplets $(\lambda_i,\mu_i,\rho_i)$ satisfy \eqref{eq:strong-convexity}.
\item[(d)]
$(\overline{\Omega},g_{1,p})$ is a simple Riemannian manifold.
\item[(e)]
$(\overline{\Omega},g_{1,s})$ is simple and admits a strictly convex foliation.
\end{itemize}
Then
\begin{equation}
\lambda_1 = \lambda_2,
\quad 
\mu_1 = \mu_2,
\quad
\text{and}
\quad
\rho_1 = \rho_2.
\end{equation}
\end{theorem}

A few points about the nature of our assumptions is in order:
\begin{itemize}
\item
Asymmetric assumptions between models:
\begin{itemize}
    \item We make stringent geometric assumptions on model 1 (i.e. the Lamé triplet $(\lambda_1,\mu_1,\rho_1)$), while no geometric assumptions whatsoever are made on model 2.
    \item Similar asymmetric rigidity results exist in the literature.
    In \cite[Theorem 1.4.]{stefanov2016boundary} only one of the two metrics compared is assumed to admit a strictly convex foliation.
    In \cite{croke1991rigidity} (which we use as lemma~\ref{lem:Cro}) and \cite{wen2015simple} only one of the two compared Riemannian metrics is assumed to be simple.
\end{itemize}
\item
Asymmetric assumptions between polarizations:
Our geometric assumptions are different for $p$ and $s$ wave metrics in model 1.
\item
Tradeoff between measurement time and geometric assumptions:
\begin{itemize}
    \item Our measurement time $T$ is only assumed to exceed the $g_{1,s}$-diameter of the domain in lemma~\ref{lem:Rac2}, not the lengths of all geodesics.
    Once we assume simplicity in theorem~\ref{thm:general}, these two quantities agree.
    \item Our reduction from PDE data to geometric data produces boundary distance data, not lens data (i.e. scattering relation and lengths of geodesics; see e.g.~\cite{stefanov2016boundary}).
    Lens data is compatible with the foliation condition while boundary distance data is not, which is why we make the additional assumption of simplicity of $g_{1,p/s}$.
    \item It is possible to formulate a similar result to theorem~\ref{thm:general} with foliation conditions only on the metrics $g_{1,p/s}$. 
    The assumptions differ from ours in a few ways:
    The measurement time is longer than all $g_{1,s}$-geodesics, not just greater than the $g_{1,s}$-diameter.
    Both metrics $g_{1,p/s}$ are assumed to have a strictly convex foliation, $\p \Omega$ is strictly convex with respect to $g_{1,p/s}$, and neither metric is assumed simple.
    Here is a sketchy argument to support this:
    The overall idea is similar to the one we give for our main results in this paper, making use of \cite{rachele2000boundary}, \cite[Theorem 1.4.]{stefanov2016boundary}, 
    \cite[Theorems 2]{stefanov2017elastic},  and \cite[Theorem 1]{zhai2025determination}. 
    However, in order to recover the lens data instead of the boundary distance function, one needs to replace our lemma~\ref{lem:Rac2} with  \cite[Theorems 2]{stefanov2017elastic}.
    To our best knowledge, this result has not appeared in the literature yet.
\end{itemize}
\end{itemize}

Our main result is analogous to many well-known rigidity results in geometry. For instance, it was proved in \cite{gromov1983filling} that the Euclidean metric is boundary distance rigid. This means that if any Riemannian metric on a bounded Euclidean domain with smooth boundary has the same distances between all boundary points as the Euclidean metric, then this metric is Euclidean up to some boundary-preserving diffeomorphism. 
On the other hand, it was proved in \cite[Theorem A]{croke1991rigidity} that any complete Riemannian metric without conjugate points on $\R^n$ which is isometric to the Euclidean metric outside a compact set must be isometric to the Euclidean metric. One should read these classical results as ``if it looks flat, it is flat". Particularly, no additional assumptions are imposed for the second metric. This is also the case in our work.
It is well known that not all manifolds are boundary rigid, and there are non-isometric manifolds with a common boundary distance function. This will easily happen if a manifold has trapped geodesics or if all distance-minimizing geodesics between boundary points avoid some interior region of the manifold.

Since the days of \cite{gromov1983filling}, the study of the boundary rigidity problem has seen many breakthroughs. It was conjectured in~\cite{michel1981rigidite} that all simple Riemannian manifolds are boundary rigid.
In two dimensions, this was verified in~\cite{pestov2005two}.
In higher dimensions, the boundary rigidity problem is still open. However, the boundary distance rigidity holds in a conformal class of simple metrics \cite{croke1991rigidity}, while a local version of this result was presented in \cite{stefanov2016boundary}. Metrics near Euclidean are boundary rigid \cite{burago2010boundary}, a generic simple metric is boundary rigid \cite{SU3}. In their groundbreaking paper \cite{stefanov2017local}, Stefanov--Uhlmann--Vasy prove that any three or higher-dimensional Riemannian manifold that admits a strictly convex function is boundary rigid. 
It is not known if all simple $n$-manifolds for $n\geq 3$ satisfy the strictly convex foliation condition. This is a major open problem in geometric inverse problems since establishing this would, in combination with \cite{stefanov2017local}, settle Michel's conjecture~\cite{michel1981rigidite}. 
We would also like to mention the recent survey \cite{survey}, and the excellent textbook \cite{paternain2023book} on this topic.

Very recently, Oksanen--Rakesh--Salo proved in \cite{oksanen2024rigidity} the analogues rigidity result for the Lorentzian Calder\'on problem. In this work, the authors showed that if a globally hyperbolic Lorentzian metric in $\R^{n+1}$ agrees with the Minkowski metric outside a compact set and has the same DN-map as the Minkowski metric then this metric agrees with the Minkowski metric modulo some diffeomorphism of $\R^{n+1}$ that agrees with the identity outside a larger compact set. Again, in this paper, no additional assumptions are imposed for the second metric. We would like to point out that in the elliptic case, the Calder\'on problem --- which asks whether the elliptic DN-map determines a Riemannian manifold --- such a rigidity result is not known. Actually, even the local version, where the second metric is assumed to be close to the Euclidean metric, remains unsolved.

There is a vast literature concerning the classical formulation of the Lorentzian Calder\'on problem. In this setting, one assumes that the Lorentzian manifold $(M, g)$ is ultrastatic (or a product), that is, $M = \R \times N$, and $g(t, x) = -dt^2 + h(x)$ where $(N, h)$ is a compact Riemannian manifold with smooth boundary. In this setting, the wave operator is 
$
\Box_g:=\p_t^2-\Delta_g,
$
and the question posed above is equivalent to the Gel'fand problem \cite{gel1954some} or the inverse boundary spectral problem \cite{Katchalov2001}. This problem asks if the boundary spectral data, eigenvalues, and Neumann traces of the Dirichlet eigenfunctions of the Laplacian determine a compact Riemannian manifold with boundary up to an isometry.
The Boundary Control method (BC-method), introduced in \cite{belishev1987approach, belishev1992reconstruction} and based on the time-optimal unique continuation theorem \cite{Tataru}, shows that the hyperbolic DN-map determines $(N,h)$ up to a diffeomorphism. There
are many further developments of the BC-method. See for instance \cite{helin2018correlation, Katchalov2001, kurylev2018inverse,  lassas2025hyperbolic, lassas2024disjoint, lassas2014inverse, saksala2025inverse}.
The unique continuation theorem remains valid for metrics $g(t, x)$ depending real-analytically on $t$, and in this case, a version
of the BC-method has been developed in \cite{eskin2007inverse}. However, the BC-method does not extend to general Lorentzian metrics, since the time-optimal unique continuation theorem fails in general \cite{Alinhac, Alinhac_Baoendi}. 
The recent breakthrough works \cite{Alexakis_Feiz_Oksanen_22, Alexakis_Feiz_Oksanen_23} make progress in the Lorentzian Calder\'on problem in a fixed conformal class by proving a time-optimal unique continuation theorem where real-analyticity is replaced by
Lorentzian curvature bounds.

In comparison to the acoustic hyperbolic inverse problem, which was discussed above, only very little is known about its counterpart in linear isotropic elasticity.
The difficulty in solving these elastic problems lies in their nature of being hyperbolic systems carrying several wave speeds. Even though, Holmgren–John–Tataru unique continuation results are known to be true for the Lam\'e systems \cite{eller2002uniqueness}, such as \eqref{eq:bvp}, we lack the boundary controllability of the elastic waves in the subdomain filled with the faster $p$-waves, while such a control exists in the smaller subdomain filled with the slower $s$-waves \cite{belishev2002dynamical}. Thus, the BC-method is not applicable for the hyperbolic inverse problem for Lam\'e systems.   
To the best of our knowledge, all positive results for Lam\'e systems have been obtained under major geometric constraints. It is worth mentioning that the acoustic analog of this inverse problem was solved in \cite{belishev1992reconstruction} without any geometric assumptions. Next, we survey literature for the hyperbolic inverse problem in linear elasticity.

In \cite{belishev2006dynamical} it was shown that if $\mu_i$ are constants, then the wave speeds are determined by the hyperbolic DN-map. This result is based on wave splitting, due to which the problem can be reduced to the inverse problems for the acoustical and Maxwell subsystems. 
In a series of three papers \cite{rachele2000boundary, rachele2000inverse, rachele2003uniqueness}, Rachele proved that if both metrics $g_p$ and $g_s$ are simple and the sectional curvature of $g_p$ is at most only slightly positive, then the hyperbolic DN-map determines the Lam\'e triplet $(\lambda,\mu,\rho)$ outside the locus $\lambda=2\mu$.
In contrast to Rachele's symmetric assumption on the two models, we assume that one of the Lam\'e triplets is both simple and satisfies the convex foliation condition, while we impose no geometric assumptions for the second triplet whatsoever.

During the last decade, dating back to the breakthrough results of Uhlmann--Vasy \cite{uhlmann2016inverse} and Stefanov--Uhlmann--Vasy \cite{stefanov2017local}, it has been popular to study geometric inverse problems under the strict convex foliation condition. Under this geometric assumption, it was shown in \cite{stefanov2017elastic} that the knowledge of the hyperbolic DN-map uniquely determines the $p$-wave speed $c_p$, if there is a strictly convex foliation with respect to it. In this paper, a similar result was also presented for the $s$-wave speed $c_s$. Analogously to the simple geometries, \cite{bhattacharyya2018local} extended \cite{stefanov2017elastic} to the recovery of the whole Lam\'e triplet $(\lambda,\mu,\rho)$ outside the locus $\lambda=2\mu$. 
Very recently, Zhai extended in \cite{zhai2025determination} this result to the determination of the Lam\'e triplet $(\lambda,\mu,\rho)$ while only requiring the foliation condition for the $s$-wave speed and not requiring $\lambda\neq 2\mu$.

As mentioned above in isotropic elasticity, there are two kinds of waves: $p$- and $s$-waves, with  $s$-waves corresponding (in spatial dimension 3) to a multiplicity 2 (and $p$-waves a simple) eigenvalue of the Christoffel matrix
$
\Gamma_{i\ell}(x,\xi):= \rho^{-1}c_{ijk\ell}\xi_j \xi_k,
$ 
for $x \in \Omega$ and $\xi \in \R^3\setminus\{0\}$. 
In anisotropic elasticity, typically the $s$-wave eigen space is broken up, at least in most parts of the cotangent bundle. Thus, there are three waves. The $qp$ waves, as well as the $qs_1$ and $qs_2$ waves, with the latter corresponding to the ‘breaking up’ of the $s$-waves.

Even less is known about the recovery of anisotropic elastic coefficients from the respective hyperbolic DN-map. Up to the best of our knowledge, the state-of-the-art results in this regime are \cite{mazzucato2007uniqueness, de2019unique, de2020recovery, zou2024partial}. 
In \cite{mazzucato2007uniqueness}, the authors study general anisotropic elastic media that have a disjoint wave mode, and extend results from microlocal analysis \cite{rachele2000boundary, rachele2000inverse, rachele2003uniqueness} to describe the propagation of singularities for the disjoint mode. Applying these results, they showed that the DN-map uniquely determines the travel time between boundary points for the disjoint mode. In a general anisotropy, the disjoint wave mode does not need to correspond to a Riemannian metric but rather a Finsler metric \cite{de2019inverse}. It is known that even simple Finsler manifolds are not boundary rigid \cite{ivanov2013local}, but it is not known that this would be the case for Finsler metrics arising from linear elasticity.
Paper \cite{de2019unique} studies the recovery of the piecewise analytic density and stiffness tensor of a three-dimensional domain from the hyperbolic DN-map. In this article, a global uniqueness result is presented if the medium is (1) transversely isotropic, there is an axis of symmetry such that in the plane orthogonal to this axis, the elastic tensor is isotropic; the elastic tensor has five independent coefficients instead of the 21 in the fully general isotropic case, or (2) orthorhombic, three mutually orthogonal planes of symmetry; the elastic tensor has nine independent coefficients.
The authors also obtain the uniqueness of a fully anisotropic stiffness tensor, assuming that it is piecewise constant.
The last two papers focus on a transversely isotropic medium. In the older \cite{de2020recovery} of these two the authors show that in transversely isotropic media, under appropriate convexity conditions, knowledge of the $qs_1$ wave travel times determines the tilt of the axis of isotropy as well as some of the elastic material parameters, and the knowledge of $qp$ and $qs_2$ travel times conditionally determines a subset of the remaining parameters, in the sense that if some of the remaining parameters are known, the rest are determined, or if the remaining parameters satisfy a suitable relation, they are all determined, under certain non-degeneracy conditions. The newer of the two papers \cite{zou2024partial} presents stability estimates for recovering either one parameter from one wave speed or two parameters from two wave speeds with the remaining parameters either known or with a known functional relationship. In particular, these estimates provide injectivity among parameters that differ on sets of small width.

\subsection*{Acknowledgments}

JI was supported by the Research Council of Finland (Flagship of Advanced Mathematics for Sensing Imaging and Modelling grant 359208; Centre of Excellence of Inverse Modelling and Imaging grant 353092; and other grants 351665, 351656, 358047, 360434) and a Väisälä project grant by the Finnish Academy of Science and Letters.
TS was partially supported by the National Science Foundation (DMS-2510272) and the  Simons Foundation Travel Support for Mathematicians (MPS-TSM-00013291). LY is partially supported by the AMS-Simons Travel Grant.

\section{Proofs}

Our proof combines ideas from
\cite{croke1991rigidity, rachele2000boundary,rachele2000inverse,zhai2025determination}.
We will be working with two Lamé triplets $(\lambda_i,\mu_i,\rho_i)$, $i=1,2$, and the associated DN-maps $\DN_1$ and $\DN_2$ in most of our lemmas.
We begin the discussion of the steps of the proof of Theorem~\ref{thm:general} by recalling the following boundary determination result.

\begin{lemma}[{\cite[Theorem 1]{rachele2000boundary}}]
\label{lem:Rac1}

Let $\Omega \subset \R^3$ be a bounded domain with smooth boundary.
Suppose $(\lambda_i,\mu_i,\rho_i)$, $i=1,2$, are two smooth Lam\'e triplets satisfying \eqref{eq:strong-convexity}. 
If $\DN_1=\DN_2$, 
then the following holds on $\p \Omega$:
\begin{equation}
\p_\nu^k \lambda_1= \p_\nu^k \lambda_2,
\quad 
\p_\nu^k \mu_1= \p_\nu^k \mu_2,
\quad \text{ and } \quad 
\p_\nu^k \rho_1= \p_\nu^k \rho_2,
\quad 
\text{for all } k \in \{0,1,\ldots\}.
\end{equation} 
\end{lemma}

Then we prove that the $p$-wave, as well as the $s$-wave, boundary distance functions agree.

\begin{lemma}
\label{lem:Rac2}
Let $\Omega \subset \R^n$, for $n\geq 2$, be an open and bounded domain with smooth boundary.
Suppose $(\lambda_i,\mu_i,\rho_i)$ for $i=1,2$ are two smooth Lam\'e triplets in~$\overline{\Omega}$ satisfying \eqref{eq:strong-convexity}. 
Make the following assumptions:
\begin{itemize}
\item $\DN_1=\DN_2$;
\item On $\p\Omega$ for all $k \in \{0,1,\ldots\}$
\begin{equation}
\p_\nu^k \lambda_1= \p_\nu^k \lambda_2,
\quad 
\p_\nu^k \mu_1= \p_\nu^k \mu_2
\quad \text{ and } \quad 
\p_\nu^k \rho_1= \p_\nu^k \rho_2;
\end{equation} 
\item $\p\Omega$ is strictly convex with respect to Riemannian metrics $g_{1,p}$ and $g_{1,s}$;
\item the distances between all pairs of points on $\partial\Omega$ are less than $T$ with respect to both metrics $g_{1,p}$ and $g_{1,s}$.

\end{itemize}
Then
\begin{equation}
d_{1,p}(z,w)=d_{2,p}(z,w)
\quad
\text{and}
\quad
d_{1,s}(z,w)=d_{2,s}(z,w)
\end{equation}
for all $z,w \in \partial \Omega$.
\end{lemma}

We will return to the proof of the preceding lemma in a moment, but we will cover some ideas needed in the proof.

\begin{lemma}
\label{lma:sff}
Let $\Omega \subset \R^n$, for $n\geq 2$, be an open and bounded domain with smooth boundary.
Suppose $(\lambda_i,\mu_i,\rho_i)$ for $i=1,2$ are two smooth Lam\'e triplets in~$\overline{\Omega}$.

If the coefficients $\lambda_i,\mu_i,\rho_i$ and their first order normal derivatives agree between $i=1$ and $i=2$ on $\partial\Omega$, then the second fundamental forms of the boundary with respect to $g_{1,p}$ and $g_{2,p}$ agree, as do the ones with respect to $g_{1,s}$ and $g_{2,s}$.
\end{lemma}

\begin{proof}
As all our metrics are conformally Euclidean, it suffices to prove that, the claim of this lemma follows from showing that a Riemannian metric $g$ and its conformal multiple $\hat g=e^{2\phi}g$ give rise to the same second fundamental form on $\partial M$ if $\phi|_{\partial M}=\partial_\nu \phi|_{\partial M}=0$.

For the metric $\hat g$ the unit normal field is $\hat\nu=e^{-\phi}\nu$ and according to \cite[Proposition 7.29]{lee2018introduction} the conformally changed Levi-Civita connection is

\begin{equation}
\hat\nabla_V W
=
\nabla_VW
+\der\phi(V)W
+\der\phi(W)V
-g(V,W)\nabla\phi
.
\end{equation}
Since $\phi$ and its normal derivative vanish along the boundary, we have that $\der\phi$ must also vanish along the boundary. By choosing $V$ pointing along the boundary, we now get from the previous equation that
\begin{equation}
\begin{split}
\hat SV
&=
-\hat\nabla_V\hat\nu
\\&=
-\hat\nabla_V(e^{-\phi}\nu)
\\&=
-e^{-\phi}\hat\nabla_V\nu
+e^{-\phi}\der\phi(V)\nu
\\&=
-e^{-\phi}\nabla_V\nu
-e^{-\phi}\der\phi(\nu)V
+e^{-\phi}g(V,\nu)\nabla\phi
\\&=
e^{-\phi}SV
-e^{-\phi}\der\phi(\nu)V
\\&=
SV
.
\end{split}
\end{equation}
Therefore the shape operators of $g$ and $\hat g$ satisfy
$S=\hat S$.
The second fundamental form is $(V, W)\mapsto g(V, SW)$, so this is also invariant under the kind of conformal change we have.
\end{proof}

Let us review the ideas of the proofs of \cite[Theorem 1]{rachele2000inverse} and \cite[Theorem 1]{mazzucato2007uniqueness} in the context of our Lemma~\ref{lem:Rac2}.
Because the two models for $i=1,2$ agree to infinite order on the boundary, the Lamé triplets have a common smooth extension from $\Omega$ to $\R^n$.
We denote the extended quantities and associated operators by the original symbols.

Consider now for both $i=1,2$ the hyperbolic initial value problem
\begin{equation}
\label{eq:src-version}
\begin{cases}
P_i\hat u_i = 0 & \text{in } (0,T)\times\R^n,\\
\hat u_i(0,\cdot)= s & \text{in }\R^n,\\
\partial_t \hat u_i(0,\cdot) = r & \text{in }\R^n.
\end{cases}
\end{equation}
For distributional initial conditions $(s,r)$ supported outside $\overline{\Omega}$, it follows from the equality of the DN-maps that the solutions agree outside the unknown domain:
\begin{equation}
\label{eq:exterior-agreement}
\hat u_1=\hat u_2
\text{ in }
(0,T)\times(\R^n\setminus\overline\Omega)
.
\end{equation}
For the proof of this claim, we refer to \cite[Theorem 3.1]{rachele2000inverse} and \cite[Proposition 2.15]{sylvester1991inverse} for the elastic and acoustic cases, respectively.
To make use of this, we need to choose initial conditions $(s,r)$ in~\eqref{eq:exterior-agreement} that allow for geometric interpretation.

Let $\hat z\in\R^n\setminus\overline{\Omega}$. For model $i$, we choose a bicharacteristic curve $\Gamma_i$ of the operator $P_i$ that passes through $(0,\hat z)$. Since, $g_{1,p/s}=g_{2,p/s}$ outside $\Omega$ we can with out loss of generality assume that $\Gamma_1=\Gamma_2$ for short times. Then we choose the initial conditions $(s,r)$ so that the wavefront set of $\hat u_i$  coincides with the conic closure of the bicharacteristic curve $\Gamma_i$. Note that these initial conditions are the same for both models. We refer to \cite[Lemma 4.3]{rachele2000inverse} for exact formulation and justification. As this result is built on top of a global parametrix construction for the Cauchy problem \eqref{eq:src-version}, global propagation of singularities property for real principal type systems \cite[Theorem~4.2]{dencker1982propagation}, and an existence of a distribution whose wavefront set agrees with a half-line in the fiber on top of $\hat z$ of the cotangent bundle, it does not rely on simplicity or any other geometric assumptions.

By equation~\eqref{eq:exterior-agreement} the wavefront sets of $\hat u_1$ and $\hat u_2$ agree outside $\overline{\Omega}$. Thus, we get as in \cite[Theorem 4]{rachele2000inverse} that the bicharacteristic curves $\Gamma_1$ and $\Gamma_2$ that start at $(0,\hat z)$ agree outside the set $(0,T)\times \overline{\Omega}$. Due to \cite[Lemma 4.1.]{rachele2000inverse} the spatial projection of the bicharacteristic curve $\Gamma_i$ are geodesics of the metric $g_{i,p/s}$. Therefore, the geodesics of $g_{1,p/s}$ and $g_{2,p/s}$ agree outside $\overline{\Omega}$ even if they would pass through $\Omega$. We do not claim that all geodesics that enter $\Omega$ will exit $\Omega$. Neither is this necessary for the arguments on top which our proof will rely on from here onward.
After this point, our proof of Lemma~\ref{lem:Rac2} diverges from the previous ones; the modification is necessary to ensure that distance is only measured using geodesics through $\Omega$, avoiding shortcuts through the extension.
As mentioned after Theorem~\ref{thm:general}, our proof and those in~\cite{rachele2000inverse,stefanov2017elastic} differ also in the measurement time assumption.

\begin{proof}[Proof of Lemma~\ref{lem:Rac2}]
Fix any $z\in\partial\Omega$ and take $\hat z\in\R^n\setminus\overline{\Omega}$ so that there is a $p / s$-wave geodesic from $\hat z$ to $z$ that does not meet $\Omega$.
We will vary $\hat z$ later in the argument.
Denote this geodesic by $\gamma_{\hat zz}^{p/s}$ and parametrize it so that it is at $\hat z$ at $t=0$.
Denote by $\delta_{\hat zz}^{p/s}$ the length of the segment from $\hat z$ to $z$.
We may, of course, extend this geodesic forward as a geodesic with respect to $g_{i,p/s}$ and denote the extension by $\gamma_{\hat zz}^{i,p/s}$.

By the argument above, we know that
\begin{equation}
\label{eq:exterior-agreement-geodesic}
\gamma_{\hat zz}^{1,p/s}(t)
=
\gamma_{\hat zz}^{2,p/s}(t)
\quad
\text{whenever }
\gamma_{\hat zz}^{1,p/s}(t)\notin\overline{\Omega}
\text{ and }
t<T.
\end{equation}

Let $t=\tau_{\hat zz}^{i,p/s}$ be the first time, finite or infinite, after $t=\delta_{\hat zz}^{p/s}$ when $\gamma_{\hat zz}^{i,p/s}(t)\in\partial\Omega$.
Let $e_{\hat zz}^{i,p/s}=\gamma_{\hat zz}^{i,p/s}(\tau_{\hat zz}^{i,p/s})\in\partial\Omega$ denote the corresponding exit point, only defined when $\tau_{\hat zz}^{i,p/s}<\infty$.

By assumption $\partial\Omega$ is strictly convex with respect to $g_{1,p/s}$.
Due to Lemma~\ref{lma:sff}, it is also strictly convex with respect to $g_{2,p/s}$.
Therefore the geodesics $\gamma_{\hat zz}^{i,p/s}(t)$ exit $\overline{\Omega}$ right after $t=\tau_{\hat zz}^{i,p/s}$.
Therefore~\eqref{eq:exterior-agreement-geodesic} implies that
\begin{equation}
\label{eq:exit-agreement}
\tau_{\hat zz}^{1,p/s}
=
\tau_{\hat zz}^{2,p/s}
\quad\text{and}\quad
e_{\hat zz}^{1,p/s}
=
e_{\hat zz}^{2,p/s}
\end{equation}
whenever $\tau_{\hat zz}^{1,p/s}<T$.
The time it takes to go from $z$ to $e_{\hat zz}^{i,p/s}$ along $\gamma_{\hat zz}^{i,p/s}(t)$ is $\tau_{\hat zz}^{i,p/s}-\delta_{\hat zz}^{p/s}$.

Pick any $w\in\partial\Omega\setminus\{z\}$.
The $g_{i,p/s}$-distance measured within $\overline{\Omega}$ (only involving geodesics staying in this set) between boundary points can be computed as
\begin{equation}
\label{eq:distance-from-minimization}
d_{i,p/s}(z,w)
=
\min
\{
\tau_{\hat zz}^{i,p/s}-\delta_{\hat zz}^{p/s};
e_{\hat zz}^{i,p/s}=w
\text{ and }
\hat z\text{ is as above}
\}
.
\end{equation}
Because $(\overline{\Omega},g_{i,p/s})$ is a compact Riemannian manifold with strictly convex boundary, for every $w$ there is indeed a $\hat z$ (which we may effectively think of as an initial direction from $z$) so that $e_{\hat zz}^{i,p/s}=w$.
There may, a priori, be several $g_{i,p/s}$-geodesics connecting $z$ to $w$, so taking the minimum is necessary.
Every length-minimizing curve between $z$ and $w$ in $\overline{\Omega}$ is in the interior apart from its endpoints due to strict convexity.
By moving the initial point $\hat z$ closer to $z$ along $\gamma_{\hat zz}^{p/s}$ the time difference $\tau_{\hat zz}^{i,p/s}-\delta_{\hat zz}^{p/s}$ stays put but $\delta_{\hat zz}^{p/s}$ approaches zero.

By assumption $d_{1,p/s}(z,w)<T$, we get from the equation~\eqref{eq:distance-from-minimization} that there is $\hat z$ so that
$d_{i,p/s}(z,w)
=
\tau_{\hat zz}^{i,p/s}-\delta_{\hat zz}^{p/s}$
and $\tau_{\hat zz}^{i,p/s}<T$.
Therefore combining~\eqref{eq:exit-agreement} and~\eqref{eq:distance-from-minimization} leads to
\begin{equation}
d_{1,p/s}(z,w)
=
d_{2,p/s}(z,w)
,
\end{equation}
with the $g_{i,p/s}$-distance measured only using geodesics staying in the set $\overline{\Omega}$.
The points $z,w\in\partial\Omega$ were arbitrary, so this completes the proof.
\end{proof}

In the scope of the previous theorem, there may, in principle, be later arrivals ($t>T$) for either $i$, but they have no effect on the first arrival times that are encoded in the distance functions.
Later arrivals are ruled out a posteriori when we know that the metrics $g_{2,p/s}$ are also simple when the metrics $g_{1,p/s}$ are simple. 
Lemma~\ref{lem:Rac2} works without any strong geometric assumption like simplicity, though, only assuming convexity of the boundary. This lemma generalizes a similar result presented in \cite{rachele2000inverse}.

We want to underline the fact that the distances in Lemma~\ref{lem:Rac2} are measured only within $\overline{\Omega}$.
Therefore, the choice of extension of the material parameters to $\R^n$ does not matter as long as the wave equation remains hyperbolic.
If one works with distances with respect to the extended model, one has to be careful to rule out shortcuts going around $\Omega$, either by studying only bicharacteristics that meet $\Omega$ or choosing the extension in a specific way.
This is why we formulated our proof above by fixing $z$ and pivoting $\hat z$ around it, as it ensures that we only study geodesics that actually enter $\Omega$.
This global geometric aspect of the argument was not presented or necessary in~\cite{rachele2000inverse}, as there it was assumed that all metrics $g_{i,p/s}$ are simple. For our main result, a more precise argument like the one above is required.

We would like to recall the following result stating that simplicity can be read directly from the boundary distance function.
To the best of our knowledge, a similar result is not known for the foliation condition.  

\begin{remark}
\label{rmk:simplicity}
A compact Riemannian manifold $(M,g)$ with strictly convex boundary is simple if and only if the boundary distance function $d_g\vert_{\partial M \times \partial M}$ is smooth away from the diagonal~\cite[Proposition 3.8.14]{PaternainSaloUhlmann2023}.
We are working with conformally Euclidean metrics that enjoy the rigidity of the following lemma, so we do not directly need this characterization of simplicity in our proof.
\end{remark}

\begin{lemma}[{\cite[Theorem C]{croke1991rigidity}}]
\label{lem:Cro}
If $g$ is a Riemannian metric on $\Omega$ and $c_1(x),c_2(x)>0$ on $\Omega$ are such that $g_i(x)=c_i(x)g(x)$ is simple for $i=1$ and $d_1(z,w)=d_2(z,w)$
for all $z,w \in \partial \Omega$, then $c_1\equiv c_2$.
\end{lemma}

So far, we have only been building on the simplicity assumption. 
However, it is possible that one or both of the equations 
    \begin{equation}
        \label{eq:locus}
    \lambda_i=2\mu_i, \quad \text{ for } i \in \{1,2\}
    \end{equation}
are true globally or locally. Either of these two equations, as well as large positive sectional curvature, would push our setting outside the scope of \cite{rachele2003uniqueness}, which proves the result $\rho_1=\rho_2$  outside the locus \eqref{eq:locus}, under the assumption that the simple $p$-wave metric has at most small positive curvature. Therefore, from here onward, we need to rely on the convex foliation condition and follow the steps of \cite{zhai2025determination}. 

\begin{lemma}[{\cite[Theorem 1]{zhai2025determination}}]
\label{lem:Zhai}
Let $\Omega \subset \R^3$ be an open and bounded domain with smooth boundary.
Suppose $(\lambda_i,\mu_i,\rho_i)$, $i=1,2$, are two smooth Lam\'e triplets. 
Suppose 
\begin{equation}
\DN_1=\DN_2
\end{equation}
and
\begin{equation}
c_p:=c_{1,p}=c_{2,p}
\quad \text{and} \quad 
c_s:=c_{1,2}=c_{2,s}.
\end{equation}
If $\p \Omega$ is strictly convex with respect to $g_s$, and $(\Omega, g_{s})$ admits a strictly convex function, and if $T>0$ is greater than the length of all geodesics in $(\Omega, g_{s})$ then
\begin{equation}
\rho_1=\rho_2.
\end{equation}
\end{lemma}

We are ready to prove our main theorem.

\begin{proof}[Proof of Theorem \ref{thm:general}]
Due to Lemma~\ref{lem:Rac1} we know that for all $k =0,\,1,\dots$, we have on $\p\Omega$ that
\begin{equation}
\label{eq:infinite-agreement}
\p_\nu^k \lambda_1= \p_\nu^k \lambda_2,
\quad 
\p_\nu^k \mu_1= \p_\nu^k \mu_2
\quad\text{and}\quad
\p_\nu^k \rho_1= \p_\nu^k \rho_2.
\end{equation}
Therefore, we can take 
\begin{equation}
\hat g_p=c_{1,p}^{-2}e, \quad \text{ and } \quad \hat g_s=c_{1,s}^{-2}e
\end{equation}
to be a common smooth extension of $g_{i,p}$ and $g_{i,s}$ for both $i\in \{1,2\}$ from $\Omega$ to $\R^n$ respectively. 

Since $g_{1,p}$ and $g_{1,s}$ are simple, the set $\p \Omega$ is strictly convex with respect to both of them, which is equivalent to the positive definiteness of the respective second fundamental forms.
By Lemma~\ref{lma:sff} we have shown $\p \Omega$ to be also strictly convex with respect to metrics $g_{2,p}$ and $g_{2,s}$. 

Hence, Lemma~\ref{lem:Rac2} implies the distance functions agree on the boundary:
\begin{equation}
\label{eq:bound_dist_func_agree}
d_{g_{1,p}}
=
d_{g_{2,p}}
\quad\text{and}\quad
d_{g_{1,s}}
=
d_{g_{2,s}}
\quad\text{on }
\partial\Omega\times\partial\Omega,
\end{equation}
and we get from Lemma~\ref{lem:Cro} that
\begin{equation}
c_{1,p}=c_{2,p}
\quad 
\text{and}
\quad 
c_{1,s}=c_{2,s}.
\end{equation}

By the assumption, the metric $g_{1,s}=g_{2,s}$ admits a strictly convex function, and we get from Lemma~\ref{lem:Zhai} that the densities agree. In other words
$
\rho_1=\rho_2.
$
Since we also know that 
\begin{equation}
\sqrt{\frac{\mu_1}{\rho_1}}
=
c_{1,s}
=
c_{2,s}
=
\sqrt{\frac{\mu_2}{\rho_2}}
\quad
\text{and} 
\quad
\sqrt{\frac{\lambda_1+2\mu_1}{\rho_1}}
=
c_{1,p}
=
c_{2,p}
=
\sqrt{\frac{\lambda_2+2\mu_2}{\rho_2}},
\end{equation}
we have proved that
$\lambda_1=\lambda_2$
and
$\mu_1=\mu_2$. The proof is complete.
\end{proof}

\bibliography{bibliography}
\bibliographystyle{abbrv}

\end{document}